\documentclass[11pt]{article}

\usepackage{amsmath,amssymb,amsbsy,amsfonts,amsthm,latexsym,
            amsopn,amstext,amsxtra,amscd,stmaryrd,fullpage}

\usepackage[mathscr]{eucal}
\usepackage{color}
\usepackage{comment}
\begin{document}

\newtheorem{theorem}{Theorem}
\newtheorem{lemma}[theorem]{Lemma}
\newtheorem{corollary}[theorem]{Corollary}
\newtheorem{proposition}[theorem]{Proposition}

\theoremstyle{definition}
\newtheorem*{definition}{Definition}
\newtheorem*{remark}{Remark}
\newtheorem*{example}{Example}


\def\cA{\mathcal A}
\def\cB{\mathcal B}
\def\cC{\mathcal C}
\def\cD{\mathcal D}
\def\cE{\mathcal E}
\def\cF{\mathcal F}
\def\cG{\mathcal G}
\def\cH{\mathcal H}
\def\cI{\mathcal I}
\def\cJ{\mathcal J}
\def\cK{\mathcal K}
\def\cL{\mathcal L}
\def\cM{\mathcal M}
\def\cN{\mathcal N}
\def\cO{\mathcal O}
\def\cP{\mathcal P}
\def\cQ{\mathcal Q}
\def\cR{\mathcal R}
\def\cS{\mathcal S}
\def\cU{\mathcal U}
\def\cT{\mathcal T}
\def\cV{\mathcal V}
\def\cW{\mathcal W}
\def\cX{\mathcal X}
\def\cY{\mathcal Y}
\def\cZ{\mathcal Z}


\def\sA{\mathscr A}
\def\sB{\mathscr B}
\def\sC{\mathscr C}
\def\sD{\mathscr D}
\def\sE{\mathscr E}
\def\sF{\mathscr F}
\def\sG{\mathscr G}
\def\sH{\mathscr H}
\def\sI{\mathscr I}
\def\sJ{\mathscr J}
\def\sK{\mathscr K}
\def\sL{\mathscr L}
\def\sM{\mathscr M}
\def\sN{\mathscr N}
\def\sO{\mathscr O}
\def\sP{\mathscr P}
\def\sQ{\mathscr Q}
\def\sR{\mathscr R}
\def\sS{\mathscr S}
\def\sU{\mathscr U}
\def\sT{\mathscr T}
\def\sV{\mathscr V}
\def\sW{\mathscr W}
\def\sX{\mathscr X}
\def\sY{\mathscr Y}
\def\sZ{\mathscr Z}


\def\fA{\mathfrak A}
\def\fB{\mathfrak B}
\def\fC{\mathfrak C}
\def\fD{\mathfrak D}
\def\fE{\mathfrak E}
\def\fF{\mathfrak F}
\def\fG{\mathfrak G}
\def\fH{\mathfrak H}
\def\fI{\mathfrak I}
\def\fJ{\mathfrak J}
\def\fK{\mathfrak K}
\def\fL{\mathfrak L}
\def\fM{\mathfrak M}
\def\fN{\mathfrak N}
\def\fO{\mathfrak O}
\def\fP{\mathfrak P}
\def\fQ{\mathfrak Q}
\def\fR{\mathfrak R}
\def\fS{\mathfrak S}
\def\fU{\mathfrak U}
\def\fT{\mathfrak T}
\def\fV{\mathfrak V}
\def\fW{\mathfrak W}
\def\fX{\mathfrak X}
\def\fY{\mathfrak Y}
\def\fZ{\mathfrak Z}


\def\C{{\mathbb C}}
\def\F{{\mathbb F}}
\def\K{{\mathbb K}}
\def\L{{\mathbb L}}
\def\N{{\mathbb N}}
\def\Q{{\mathbb Q}}
\def\R{{\mathbb R}}
\def\Z{{\mathbb Z}}
\def\E{{\mathbb E}}
\def\T{{\mathbb T}}
\def\P{{\mathbb P}}
\def\D{{\mathbb D}}


\def\eps{\varepsilon}
\def\mand{\qquad\mbox{and}\qquad}
\def\\{\cr}
\def\({\left(}
\def\){\right)}
\def\[{\left[}
\def\]{\right]}
\def\<{\langle}
\def\>{\rangle}
\def\fl#1{\left\lfloor#1\right\rfloor}
\def\rf#1{\left\lceil#1\right\rceil}
\def\le{\leqslant}
\def\ge{\geqslant}
\def\ds{\displaystyle}

\def\xxx{\vskip5pt\hrule\vskip5pt}
\def\yyy{\vskip5pt\hrule\vskip2pt\hrule\vskip5pt}
\def\imhere{ \xxx\centerline{\sc I'm here}\xxx }

\newcommand{\comm}[1]{\marginpar{
\vskip-\baselineskip \raggedright\footnotesize
\itshape\hrule\smallskip#1\par\smallskip\hrule}}


\def\e{\mathbf{e}}
\def\sPrc{{\displaystyle \sP_r^{(c)}}}

\title{\bf  Zeros of random orthogonal polynomials with complex Gaussian coefficients
}

\author{
{\sc Aaron M.~Yeager} \\
{Department of Mathematics, Oklahoma State University} \\
{Stillwater, OK 74078 USA} \\
{\tt aaron.yeager@okstate.edu}}

\maketitle

\begin{abstract}
Let $\{f_j\}_{j=0}^n$ be a sequence of orthonormal polynomials where the orthogonality relation is satisfied on either the real line or on the unit circle.
We study zero distribution of random linear combinations of the form
$$P_n(z)=\sum_{j=0}^n\eta_jf_j(z),$$
where $\eta_0,\dots,\eta_n$ are complex-valued i.i.d.~standard Gaussian random variables. Using the Christoffel-Darboux formula, the density function for the expected number of zeros of $P_n$ in these cases takes a very simple shape.  From these expressions, under the mere assumption that the orthogonal polynomials are from the Nevai class, we give the limiting value of the density function away from their respective sets where the orthogonality holds. In the case when $\{f_j\}$ are orthogonal polynomials on the unit circle, the density function shows that the expected number of zeros of $P_n$ are clustering near the unit circle.  To quantify this phenomenon,  we give a result that estimates the expected number of complex zeros of $P_n$ in shrinking neighborhoods of compact subsets of the unit circle.
\end{abstract}



\section{Introduction}

The study of the expected number of real zeros of polynomials $P_n(z)=\sum_{j=0}^n\eta_jz^j$ with random coefficients, called \emph{random algebraic polynomials}, dates back to the 1930's.   In 1932, Bloch and P\'{o}lya \cite{BP} showed that when $\{\eta_j\}$ are independent and identically distributed (i.i.d.)~random variables that take values from the set $\{-1,0,1\}$ with equal probabilities, the expected number of real zeros is $O(\sqrt{n})$.  Other early advancements in the subject were later made by Littlewood and Offord \cite{LO}, Kac \cite{K1}, \cite{K2}, Rice \cite{R}, Erd\H{o}s and Offord \cite{EO}, and many others.  For a nice history of the early progress in this topic, we refer the reader to the books by Bharucha-Reid and Sambandham \cite{BRS} and by Farahmand \cite{FB}.

It is common to refer to the density function for the expected number of zeros of a random polynomial as the \emph{intensity function} or the \emph{first correlation function}.  In the 1943, Kac \cite{K1} gave a formula for the intensity function of the expected number of real zeros of $P_n(z)$ when $\{\eta_j\}$ are real-valued i.i.d.~normal Gaussian coefficients.  Using that formula he was able to show that the expected number of real roots of the random algebraic polynomial is asymptotic to $2\pi^{-1}\log n$ as $n\rightarrow \infty$.  The error term in his asymptotic was further sharpened by Hammersley \cite{HM}, Wang \cite{WG}, Edelman and Kostlan \cite{EK}, and Wilkins \cite{WL}.

Remaining with case when $\{\eta_j\}$ are real-valued i.i.d.~normal Gaussian random variables, Shepp and Vanderbei gave a  formula for the intensity function for the expected number of complex zeros of the random algebraic polynomial $P_n$ in 1995.  They were also able to obtain a limit of the intensity function as $n\rightarrow \infty$.  Generalizations to other types of real-valued random variables and to other random polynomials with basis functions different than the monomials were made by Ibragimov and Zeitouni \cite{IZ}, Feildheim \cite{FL}, and Vanderbei \cite{CZRS}.

In 1996,  Farahmand \cite{F} produced a formula for the intensity function for a random algebraic polynomial when the random coefficients are complex-valued i.i.d.~standard Gaussian random variables. As an application, Farahmand  considered the spanning functions of the random polynomial to be cosine functions.  For extensions Faramand's result we refer the reader to the works  Farahmand \cite{F1}, Farahmand and Grigorash \cite{KJ} and Farahmand and Jahangiri \cite{FG}.

We will be studying a case of
the expectation of the number zeros of random polynomials of the form
\begin{equation}\label{P}
P_n(z)=\sum_{j=0}^n \eta_j f_j(z), \ \ \ \ z\in\C,
\end{equation}
where $n$ is a fixed integer, $\{f_j\}_{j=0}^n$ are entire functions real-valued on the real line,  $\eta_j=\alpha_j + i \beta_j$, $j=0,1, \dots, n$, with $\{\alpha_j\}_{j=0}^n$ and $\{\beta_j\}_{j=0}^n$ being sequences of i.i.d.~standard Gaussian random variables. The formula for the intensity function associated to $P_n$ is expressed in terms of the kernels
\begin{equation}\label{K01}
K_{n}(z,w)=\sum_{j=0}^{n}f_j(z)\overline{f_j(w)},\ \ \ \ \ \ \ K_{n}^{(0,1)}(z,w)=\sum_{j=0}^{n} f_j(z)\overline{f_j^{\prime}(w)},
\end{equation}
and
\begin{equation}\label{K2}
K_{n}^{(1,1)}(z,w)=\sum_{j=0}^{n}f_j^{\prime}(z)\overline{f_j^{\prime}(w)}.
\end{equation}
We note that since the functions $f_j(z)$ are entire functions that are real-valued on the real line, by the Schwarz Reflection Principle we have $\overline{f_j(z)}=f_j(\bar{z})$ for all $j=0,1,\dots,n$, and all $z\in \C$.

Let $N_n(\Omega)$ denote the (random) number of zeros of $P_n(z)$ as defined by \eqref{P} in a Jordan region $\Omega$ of the complex plane. Due to Edelman and Kostlan \cite{EK} (with different proofs later given by Hough, Krishnapur, Peres, and Vir$\acute{\text{a}}$g in \cite{ZGAF}, Feldheim \cite{FL}, the  author \cite{AY}, and Ledoan \cite{AL}) it is known that for each Jordan region $\Omega \subset \{z\in \C : K_{n}(z,z)\neq 0\}$, we have that the intensity function $\rho_n$ associated to $P_n$ satisfies
$$\E[N_n(\Omega)]=\int_{\Omega}\rho_n (x,y)\ dx \ dy ,$$
with
\begin{equation}
\label{thm2.1}
\rho_n(x,y)=\rho_n(z)=\frac{K_{n}^{(1,1)}(z,z)K_{n}(z,z)-\left|K_{n}^{(0,1)}(z,z)\right|^2}{\pi \left(K_{n}(z,z)\right)^2},
\end{equation}
where the kernels $K_n(z,z)$, $K_n^{(0,1)}(z,z)$, and $K_n^{(1,1)}(z,z)$, are defined in \eqref{K01} and \eqref{K2}. We note that since all the functions that make up $\rho_n$ are real valued, the function $\rho_n$ is real valued.  The function $\rho_n$ is also in fact nonnegative. Furthermore, for $(a,b)\subset \R$, it is also known the $\E[N_n(a,b)]=0$, so that $\rho_n$ does not have mass on the real line.

In the following results, we will be considering the case when the spanning functions $\{f_j\}$ of \eqref{P} are polynomials either orthogonal on the real line (OPRL), or polynomials orthogonal on the unit circle (OPUC).  We say that a collection of polynomials $\{p_j\}_{j\geq 0}$ are orthogonal on the real line with respect to $\mu$, with $\text{supp}\ \mu \subseteq \R$, if
$$\int p_n(x)p_m(x)d\mu(x)=\delta_{nm}, \ \ \ \ \text{for all $n,m\in \N \cup\{0\}$}.$$
We note that when polynomials are orthogonal on the real line, they have real coefficients, and thus are real-valued on the real line.

As mentioned, our other choice of basis of the random sum $P_n$ will the from OPUC.  These are orthogonal polynomials $\{\varphi_j\}_{j\geq 0}$ defined by a probability Borel measure $\mu$ on $\T$ such that
$$\int_{\T} \varphi_n(e^{i\theta})\overline{\varphi_m(e^{i\theta})}\ d\mu(e^{i\theta}) =\delta_{nm}, \ \ \ \ \text{for all $n,m\in \N \cup\{0\}$}.$$
When we restrict $\mu$ to be symmetric with respect to conjugation, the sequence $\{\varphi_j\}$ of OPUC will have real coefficients and consequently be real-valued on the real line.

For analogues of our results concerning random linear combinations of OPRL or OPUC with the random coefficients $\{\eta_j\}$ of $P_n$ being real-valued standard i.i.d.~Gaussian, we refer the reader to works of Das \cite{D} , Das and Bhatt \cite{DB} , Lubinsky, Pritsker, and Xie \cite{LPX} ,\cite{LPX2} (Theorems 2.2 and 2.3), and Yattselev and the author \cite{YY}.

We note there has also been work done in the higher dimensional analogs of the settings mentioned (c.f. Shiffman and Zelditch \cite{SHZ1}-\cite{SHZ3}, Bloom \cite{BL1} and \cite{BL2}, Bloom and Shiffman \cite{BLSH}, Bloom and Levenberg \cite{BLL}, and Bayraktar \cite{BY}).

Using the Christoffel-Darboux formula we show that the intensity function from \eqref{thm2.1} greatly simplifies when the spanning functions are orthogonal polynomials.

\begin{theorem}\label{OP}
Let $P_n(z)=\sum_{j=0}^n \eta_j f_j(z)$, where $\{\eta_j\}_{j=0}^n$ are complex-valued i.i.d.~standard Gaussian random variables, and $\{f_j\}_{j=0}^n$ are orthogonal polynomials.  Let $\rho_n$ be defined as in  \eqref{thm2.1}.
\begin{enumerate}
\item
When $f_j=p_j$, $j=0,\dots, n$, where the $p_j$'s are OPRL, the intensity function $\rho_n$ simplifies as
\begin{align}\label{hnOPRL}
\rho_n(z)
&=\frac{1-h_n(z)^2}{4\pi\left(\emph{\text{Im}}(z)\right)^2},\quad h_n(z)=\frac{\textup{Im}(z)|a_n^{\prime}(z)|}{\textup{Im}(a_n(z))},\quad a_n(z)=\frac{p_{n+1}(z)}{p_n(z)},
\end{align}
for $z\in \C$.
\item Let $f_j=\varphi_j$, $j=0,\dots, n$, where the $\varphi_j$'s are OPUC associated to conjugate-symmetric measure $\mu$.  When $|z|\neq 1$, the intensity function $\rho_n$ reduces to
\begin{align}\label{hnOPUC}
\rho_n(z)
&=\frac{1-|k_{n}(z)|^2}{\pi (1-|z|^2)^2},\quad k_{n}(z)=\frac{(1-|z|^2)b_{n}^{\prime}(z)}{1-|b_{n}(z)|^2},\quad b_{n}(z)=\frac{\varphi_{n+1}(z)}{\varphi_{n+1}^*(z)}.
\end{align}
where $\varphi_n^{*}(z)=z^n \overline{\varphi_n\left(\frac{1}{\bar{z}}\right)}$.
\end{enumerate}
\end{theorem}

We note that
\begin{equation*}
\text{Im}(a_n(z))=0 \iff  a_n(z)=\overline{a_n(z)}=a_n(\bar{z}) \iff  z\in \R.
\end{equation*}
Thus as written in the shape above (which is written as such for purposes of the computing the limit as $n\rightarrow \infty$), the intensity function $\rho_n$ in \eqref{hnOPRL} has singularities on the real axis due to the $\text{Im}(z)$ and $\text{Im}(a_n(z))$ in the denominators.  However these singularises exists only due to the way the intensity function is written.  For in the form of $\rho_n$ at \eqref{thm2.1}, the only potential singularity can come from when $K_n(z,z)=0$.  In the case of $f_j=p_j$, $j=0,1,\dots,n$ with $\{p_j\}$ being OPRL, it follows that $p_0(z)\neq 0$ giving $K_n(z,z)=\sum_{j=0}^n|p_j(z)|^2>0$.  Thus the intensity function \eqref{hnOPRL} is well defined and continuous everywhere on $\C$.

The restriction $|z|\neq 1$ in \eqref{hnOPUC} of Theorem \ref{OP} is present due to the use and hence assumptions of the Christoffel-Darboux formula for OPUC.  This restriction is from the fact that only when $|z|=1$ do we have $|\varphi_{n+1}^{*}(z)|=|\varphi_{n+1}(z)|$ (i.e. $|b_n(z)|=1$).  Furthermore, it is known that all the zeros of $\varphi_{n+1}(z)$ lie in $\D$, and all the zeros of $\varphi_{n+1}^{*}(z)$ are outside of $\D$.  Thus these two polynomials cannot vanish simultaneously.

Our limiting results of $\rho_n$ will be phrased in terms of assumptions on the recurrence coefficients of the orthogonal polynomials.  For a sequence $\{p_n\}$ of OPRL, the Three Term Recurrence Relation (Theorem 3.2.1 \cite{SZ}) states
\begin{equation}\label{OPRLREC}
xp_{n}(z)=a_np_{n+1}(z)+b_np_n(z)+a_{n-1}p_{n-1}(z), \quad n=1,2,\dots,
\end{equation}
where the recurrence coefficient sequences $\{a_n\}$ and $\{b_n\}$ can be given explicitly in terms of the leading coefficient of $p_n$ and $p_{n-1}$.  Due to Nevai (Theorem 13 p.~33 \cite{NV}, see also Totik p.~99 \cite{TO}), the condition that $a_n\rightarrow a$ and $b_n\rightarrow b$ as $n\rightarrow \infty$, with $a\geq 0$ and $b\in \R$, is equivalent to
\begin{equation}\label{OPRLNevai}
\lim_{n\rightarrow \infty}\frac{p_{n+1}(z)}{p_n(z)}=\frac{z-b+\sqrt{(z-b)^2-4a^2}}{2},
\end{equation}
with the convergence being valid locally uniformly for $z \notin \text{supp}\ \mu.$
When \eqref{OPRLNevai} holds for a sequence $\{p_n\}$ of OPRL, we say that the sequence is in the Nevai Class.  We note that this class is sometimes denoted as $M(a,b)$.

The Three Term Recurrence Relation (Theorem 1.5.4 \cite{BS}) for a sequence $\{\varphi_n\}$ of OPUC  says
\begin{equation}\label{OPUCREC}
\varphi_{n+1}(z)=\frac{ z \varphi_n(z)-\bar{\alpha}_n\varphi_{n}^*(z)}{\sqrt{1-|\alpha_n|^2}},\quad n=0,1,\dots
\end{equation}
where sequence of recurrence coefficients $\{\alpha_n\}\subset \D$, and $\varphi_n^*(z)=z^n\overline{\varphi_n(1/\bar{z})}$.  From the recurrence relation it can be seen that when $\{\alpha_n\}\subset (-1,1)$, the sequence $\{\varphi_n\}$ will be real-valued on the real line. Furthermore, in this case it known that there exists a unique conjugate-symmetric probability measure $\mu$ whose associated orthogonal polynomials satisfy \eqref{OPUCREC} (Theorem 1.7.11 \cite{BS}).  Hence one can refer to sequence $\{\varphi_n\}$ of OPUC  as defined by either the measure $\mu$ or the recurrence coefficients $\{\alpha_n\}$.  The ratio asymptotics (Theorem 1.7.4 of \cite{BS}) in this case are
\begin{equation}\label{OPUCNevai}
\lim_{n\rightarrow \infty}\alpha_n=0 \iff \lim_{n\rightarrow \infty} \frac{\varphi_n(z)}{\varphi_n^*(z)}=0,
\end{equation}
where the convergence holds locally uniformly for $z\in \D.$
When \eqref{OPUCNevai} holds for a sequence $\{\varphi_n\}$ of OPUC, we say that the sequence is from the Nevai Class.

\begin{corollary}\label{GOP2}
Let $P_n(z)=\sum_{j=0}^n \eta_j f_j(z)$, where $\{\eta_j\}_{j=0}^n$ are complex-valued i.i.d.~standard Gaussian random variables, and $\{f_j\}_{j=0}^n$ are orthogonal polynomials.
\begin{enumerate}
\item When $\{p_j\}$ are OPRL from the Nevai class, the intensity function $\rho_n$ from \eqref{hnOPRL} for the random orthogonal polynomial  satisfies
\begin{equation}\label{IOA}
\lim_{n\rightarrow \infty}\rho_n(z)= \frac{1}{4\pi \left(\emph{\text{Im}}(z)\right)^2}-\frac{|z-b+\sqrt{(z-b)^2-4a^2}|^2}{4\pi|(z-b)^2-4a^2|( \emph{\text{Im}}(z+\sqrt{(z-b)^2-4a^2}\ ) )^2},
\end{equation}
 locally uniformly for all $z\notin \text{supp}\ \mu$.

\item Let $\{\varphi_j\}$ be OPUC from the Nevai class such that their associated recurrence coefficients satisfy $\{\alpha_j\}\subset (-1,1)$.   Then the intensity function $\rho_n$ in \eqref{hnOPUC} for the random orthogonal polynomial  possess
\begin{equation}\label{LIOPUC}
\lim_{n\rightarrow \infty}\rho_n(z)= \frac{1}{\pi(1-|z|^2 )^2},
\end{equation}
locally uniformly for all $z\in\C\setminus \T$.
\end{enumerate}
\end{corollary}
When $a=1/2$ and $b=0$ in the definition of the Nevai class for the OPRL \eqref{OPRLNevai}, it is known that this class contains contains the Chebyshev polynomials.  The result of \eqref{IOA} extends the limiting value given by Farahmand and Grigorash (Section 4 of \cite{FG}) in which the spanning functions of their random trigonometric polynomial can be modified to be the Chebyshev polynomials.  We note that the result of \eqref{LIOPUC} extends the limiting value of the first correlation function given by Peres and Vir\'{a}g \cite{PV} (i.e.~taking $n=1$ of their Theorem 1) when the spanning functions were the monomials to that of a very general basis of OPUC. The result further extends their work in that this limiting value also holds for exterior of unit circle.

From \eqref{hnOPUC} of Theorem \ref{OP} and \eqref{LIOPUC} of Theorem \ref{GOP2} we see that the intensity function and its limiting value for the random orthogonal polynomial spanned by OPUC is singular on the unit circle.  Assuming a little more on the measure $\mu$ associated to the OPUC we can quantify how the zeros approach the unit circle.
\begin{theorem}\label{OPUCBand}
Let $P_n(z)=\sum_{j=0}^n\eta_j\varphi_j(z)$, where $\{\eta_j\}$ are complex-valued i.i.d.~standard Gaussian random variables, and $\{\varphi_j\}$ are OPUC such that their associated recurrence coefficients satisfy $\{\alpha_j\}\subset(-1,1)$ with $ \alpha_j \to 0 $ as  $ j\to\infty $. Let $ S $ be a compact subset of $ \T\setminus\{\pm1\} $. Assume, in addition, that the measure $ \mu $ associated to the sequence $\{\varphi_j\}$ is absolutely continuous with respect to the arclength measure on an open set containing $ S $ and its Radon-Nikodym derivative is positive and continuous at each point of $ S $. Given $ -\infty<\tau_1<\tau_2<\infty $, it follows that
\begin{equation}\label{BandLim}
\lim_{n\rightarrow \infty}\frac1n \E\big[N\big(\Omega(S,\tau_1,\tau_2)\big)\big] = \frac{|S|}{2\pi}\left(\frac{H^\prime(\tau_2)}{H(\tau_2)}-\frac{H^\prime(\tau_1)}{H(\tau_1)}\right),
\end{equation}
where $ \Omega(S,\tau_1,\tau_2) := \big\{ rz: z\in S, \; r\in(1+\frac{\tau_1}{2n},1+\frac{\tau_2}{2n}) \big\} $ and $ \displaystyle H(\tau) := \frac{e^\tau-1}{\tau} $.
\end{theorem}
We note that $ H^\prime/H $ is increasing on the real line with
$$
\lim_{\tau\to-\infty}\frac{H^\prime(\tau)}{H(\tau)}=0 \quad \text{and} \quad \frac{H^\prime(\tau)}{H(\tau)} =  1 - \frac{H^\prime(-\tau)}{H(-\tau)}.
$$
Thus in the setting of Theorem \ref{OPUCBand}, the zeros of a random orthogonal polynomial spanned by OPUC approaching $ S $ are expected to be contained in an annular band around $ S $ of width $ n^{-1+\epsilon} $ for any $ \epsilon>0 $.

We note that when the coefficients of the random orthogonal polynomial spanned by OPUC satisfying the conditions of Theorem \ref{OPUCBand} are real-valued i.i.d.~standard Gaussian random variables, the analog of the above result was recently proved Yattselev and the author (c.f. Theorem 1.7 of \cite{YY}).  Remarkably, both the cases of random orthogonal polynomials with real-valued or complex-valued coefficients yield the same asymptotic in \eqref{BandLim}.

\section{Proofs}

\subsection{The Intensity Function for Random Orthogonal Polynomials}

In this section we use the Christoffel-Darboux formula for OPRL and OPUC to simplify the kernels $K_{n}(z,z)$, $K_{n}^{(0,1)}(z,z)$, and $K_{n}^{(1,1)}(z,z)$ which make up the intensity function $\rho_n$ from \eqref{thm2.1}.  For convenience of the reader, we state the Christoffel-Darboux formula for OPRL  (Theorem 3.2.2, p. 43 of \cite{SZ}): for $z,w\in \C$ and $\{p_j\}_{j\geq 0}$  OPRL, with $k_j$ being the leading coefficient of $p_j$, we have
\begin{equation}\label{CD1}
\sum_{j=0}^{n}p_j(z)p_j(w)=\frac{k_{n}}{k_{n+1}}\cdot \frac{p_{n+1}(z)p_{n}(w)-p_{n}(z)p_{n+1}(w)}{z-w}, \ \  z\neq w.
\end{equation}
Furthermore, on the diagonal $z=w$ it takes the form
\begin{equation}\label{CD2}
\sum_{j=0}^{n}\left(p_j(z)\right)^2=\frac{k_{n}}{k_{n+1}}\cdot (p_{n+1}^{\prime}(z)p_{n}(z)-p_{n}^{\prime}(z)p_{n+1}(z)).
\end{equation}

For a collection of OPUC $\{\varphi_j\}_{j\geq 0}$, the Christoffel-Darboux formula for OPUC (Theorem 2.2.7, p. 124 of \cite{BS}) states that for $z,w\in \C$ with $\bar{w}z\neq 1$, we have
\begin{equation}\label{CDUP}
\sum_{j=0}^{n}\varphi_j(z)\overline{\varphi_j(w)}= \frac{\overline{\varphi_{n+1}^{*}(w)}\varphi_{n+1}^{*}(z)-\overline{\varphi_{n+1}(w)}\varphi_{n+1}(z)}{1-\bar{w}z},
\end{equation}
where $\varphi_n^{*}(z)=z^n \overline{\varphi_n\left(\frac{1}{\bar{z}}\right)}$.

Before obtaining our representations of the kernels, let us note that since the polynomials $\{p_j\}$  are orthogonal on the real line, and since we are assuming that the recurrence coefficients $\{\alpha_j\}$ associated to $\{\varphi_j\}$ satisfy $\{\alpha_j\}\subset (-1,1)$, both classes of orthogonal polynomials have real coefficients.  Thus when using conjugation we have that $\overline{p_j(z)}=p_j(\bar{z})$ and $\overline{\varphi_j(z)}=\varphi_j(\bar{z})$ for all $j=0,1,\dots$, and all $z\in \C$.

\begin{proof}[Proof of \eqref{hnOPRL} in Theorem \ref{OP}]
For $z\neq w$, taking derivatives of \eqref{CD1} yields
\begin{align}\nonumber
\sum_{j=0}^{n}p_j(z)p_j^{\prime}(w)&=\frac{k_{n}}{k_{n+1}} \left(\frac{p_{n+1}(z)p_{n}^{\prime}(w)-p_{n}(z)p_{n+1}^{\prime}(w)}{z-w}
   +\frac{p_{n+1}(z)p_{n}(w)-p_{n}(z)p_{n+1}(w) }{(z-w)^2}\right)\\
   \label{CDw}
   &=\frac{k_{n}}{k_{n+1}} \cdot\frac{p_{n+1}(z)p_{n}^{\prime}(w)-p_{n}(z)p_{n+1}^{\prime}(w)}{z-w}+\frac{\sum_{j=0}^np_j(z)p_j^{\prime}(w)}{z-w},
\end{align}
and
\begin{align}
\nonumber
\sum_{j=0}^{n}p_j^{\prime}(z)p_j^{\prime}(w)&=\frac{k_{n}}{k_{n+1}} \Bigg(\frac{p_{n+1}^{\prime}(z)p_{n}^{\prime}(w)-p_{n}^{\prime}(z)p_{n+1}^{\prime}(w)}{z-w}
   -\frac{p_{n+1}(z)p_{n}^{\prime}(w)-p_{n}(z)p_{n+1}^{\prime}(w) }{(z-w)^2}\\
   \nonumber
   & \qquad \qquad + 
   \frac{p_{n+1}^{\prime}(z)p_{n}(w)-p_{n}^{\prime}(z)p_{n+1}(w) }{(z-w)^2}
   -\frac{2\left(p_{n+1}(z)p_{n}(w)-p_{n}(z)p_{n+1}(w)\right) }{(z-w)^2}\Bigg)\\
   \label{CDwz}
   &=\frac{k_{n}}{k_{n+1}} \cdot\frac{p_{n+1}^{\prime}(z)p_{n}^{\prime}(w)-p_{n}^{\prime}(z)p_{n+1}^{\prime}(w)}{z-w}-\frac{\sum_{j=0}^np_j(z)p_j^{\prime}(w)}{z-w}+
   \frac{\sum_{j=0}^np_j^{\prime}(z)p_j(w)}{z-w}.
\end{align}

Setting $w=\bar{z}$ in \eqref{CD1}, \eqref{CDw}, and \eqref{CDwz}, since the coefficients of $\{p_j\}$ are real it follows that
\begin{align}\label{CDB01}
K_{n}(z,z)&=\sum_{j=0}^{n}p_j(z)\overline{p_j(z)}=\frac{k_{n}}{k_{n+1}}\cdot \frac{p_{n+1}(z)p_{n}(\bar{z})-p_{n}(z)p_{n+1}(\bar{z})}{2i\text{Im}(z)},\\
\label{CDB1}
K_{n}^{(0,1)}(z,z)&=\sum_{j=0}^{n} p_j(z)\overline{p_j^{\prime}(z)}
=\frac{k_{n}}{k_{n+1}} \cdot\frac{p_{n+1}(z)p_{n}^{\prime}(\bar{z})-p_{n}(z)p_{n+1}^{\prime}(\bar{z})}{2i\text{Im}(z)}+\frac{K_{n}(z,z) }{2i\text{Im}(z)},\\
\label{CDB2}
K_{n}^{(1,1)}(z,z)&=\sum_{j=0}^{n}p_j^{\prime}(z)\overline{p_j^{\prime}(z)}
=\frac{k_{n}}{k_{n+1}}\cdot\frac{\text{Im}(p_{n+1}^{\prime}(z)p_{n}^{\prime}(\bar{z}))}{\text{Im}(z)}-\frac{K_{n}^{(0,1)}(z,z)}{2i\text{Im}(z)}
+\frac{\overline{K_{n}^{(0,1)}(z,z)}}{2i\text{Im}(z)}.
\end{align}

For our representation of $K_{n}(z,\bar{z})$ we simply use \eqref{CD2} and again that the coefficients of $\{p_j\}$ are real to achieve
\begin{align}\label{CDA0}
K_{n}(z,\bar{z})&=\sum_{j=0}^{n}p_j(z)\overline{p_j(\bar{z})}=\sum_{j=0}^{n}p_j(z)p_j(z)
=\frac{k_{n}}{k_{n+1}}\left(p_{n+1}^{\prime}(z)p_{n}(z)-p_{n}^{\prime}(z)p_{n+1}(z)\right).
\end{align}

Using our derived expressions \eqref{CDB01}, \eqref{CDB1}, \eqref{CDB2}, and \eqref{CDA0}, the numerator of the intensity function $\rho_n$ from \eqref{thm2.1} simplifies as
\begin{align}
\nonumber
K_{n}^{(1,1)}&(z,z)K_{n}(z,z)-|K_{n}^{(0,1)}(z,z)|^2
=\frac{\left(K_{n}(z,z) \right)^2-\left| K_{n}(z,\bar{z}) \right|^2}{4\left(\text{Im}(z)\right)^2}.
\end{align}

Therefore, using the expression for the numerator above and recalling the relations \eqref{CDB01} and \eqref{CDA0}, we see that the intensity function given by \eqref{thm2.1} is
\begin{align*}
\rho_n(z)&=\frac{K_{n}^{(1,1)}(z,z)K_{n}(z,z)-|K_{n}^{(0,1)}(z,z)|^2}{\pi  \left(K_{n}(z,z)\right)^2} \\
&=\frac{1}{4\pi\left(\text{Im}(z)\right)^2}\left(1-\frac{\left|K_{n}(z,\bar{z})\right|^2}{\left(K_{n}(z,z)\right)^2}\right)\\
&=\frac{1}{4\pi\left(\text{Im}(z)\right)^2}\left(1-\frac{(2i \textup{Im}(z))^2\left| p_{n+1}^{\prime}(z)p_n(z)-p_n^{\prime}(z)p_{n+1}(z)\right|^2}{\left(p_{n+1}(z)p_n(\bar{z})-p_n(z)p_{n+1}(\bar{z})\right)^2}\right)\\
&=\frac{1}{4\pi \left(\text{Im}(z)\right)^2}\left( 1-\frac{(2i \text{Im}(z))^2\left|\left( \frac{p_{n+1}(z)}{p_n(z)} \right)^{\prime}\right|^2}{\left(\frac{p_{n+1}(z)}{p_n(z)}-\frac{p_{n+1}(\bar{z})}{p_n(\bar{z})}\right)^2} \right)\\
&=\frac{1-h_n(z)^2}{4\pi\left(\text{Im}(z)\right)^2},
\end{align*}
where
\begin{equation*}
h_n(z)=\frac{\text{Im}(z)|a_n^{\prime}(z)|}{\text{Im}(a_n(z))}, \quad a_n(z)=\frac{p_{n+1}(z)}{p_n(z)},
\end{equation*}
which gives the result of \eqref{hnOPRL} in Theorem \ref{OP}.
\end{proof}

\begin{proof}[Proof of \eqref{hnOPUC} in Theorem \ref{OP}]
Applying the Christoffel-Darboux formula for OPUC from \eqref{CDUP}, and making derivations analogously as done for the kernels for OPRL, our representations of $K_n(z,z)$, $K_n^{(0,1)}(z,z)$, and $K_n^{(1,1)}(z,z)$ are as follows:
\begin{align}\label{CDBOUP}
K_{n}(z,z)&=\sum_{j=0}^{n}\varphi_j(z)\overline{\varphi_j(z)} =\frac{\left|\varphi_{n+1}^{*}(z)\right|^2-\left|\varphi_{n+1}(z)\right|^2}{1-\left|z\right|^2},
\end{align}

\begin{align}\label{CDB1UP}
K_n^{(0,1)}(z,z)&=\sum_{j=0}^{n}\varphi_j(z)\overline{\varphi_j^{\prime}(z)}=\frac{\overline{\varphi_{n+1}^{*\ \prime}(z)}\varphi_{n+1}^{*}(z)-\overline{\varphi_{n+1}^{\prime}(z)}\varphi_{n+1}(z)}{1-|z|^2}
+\frac{zK_n(z,z)}{1-|z|^2},
\end{align}
and
\begin{align}\label{CDB2UP}
K_n^{(1,1)}(z,z)&=\sum_{j=0}^n|\varphi_j^{\prime}(z)|^2
=\frac{|\varphi_{n+1}^{* \ \prime}(z)|^2-|\varphi_{n+1}^{\prime}(z)|^2}{1-|z|^2}+ \frac{\bar{z}K_n^{(0,1)}(z,z)
 +z\overline{ K_n^{(0,1)}(z,z)}+K_n(z,z)}{1-|z|^2}.
\end{align}

Using \eqref{CDBOUP}, \eqref{CDB1UP}, and \eqref{CDB2UP}, the numerator of the intensity function $\rho_n$ of \eqref{thm2.1} reduces to
\begin{align*}
K_n^{(1,1)}(z,z)K_n(z,z)-&|K_n^{(0,1)}(z,z)|^2
=\frac{\left(K_n(z,z)\right)^2}{\left(1-|z|^2\right)^2}-\frac{\left|\varphi_{n+1}^*(z)\varphi_{n+1}^{\prime}(z)-\varphi_{n+1}^{*\ \prime}(z)\varphi_{n+1}(z)  \right|^2  }{ \left(1-|z|^2\right)^2  }.
\end{align*}

From the above numerator and \eqref{CDBOUP}, the intensity function at \eqref{thm2.1} becomes
\begin{align}\nonumber
\rho_n(z)&=\frac{K_{n}^{(1,1)}(z,z)K_{n}(z,z)-\left|K_{n}^{(0,1)}(z,z)\right|^2}{\pi \left(K_{n}(z,z)\right)^2} \\
\nonumber
&=\frac{1}{\pi\left(1-|z|^2 \right)^2}\left(1-\frac{\left|\varphi_{n+1}^{*}(z)\varphi_{n+1}^{\prime}(z)-\varphi_{n+1}^{*\ \prime}(z)\varphi_{n+1}(z)  \right|^2}{\left(K_n(z,z)\right)^2}\right)\\
\label{facrho}
&=\frac{1}{\pi\left(1-|z|^2 \right)^2}\left(1-\frac{(1-|z|^2)^2\left|\varphi_{n+1}^{*}(z)\varphi_{n+1}^{\prime}(z)-\varphi_{n+1}^{*\ \prime}(z)\varphi_{n+1}(z)  \right|^2}{\left(|\varphi_{n+1}(z)|^2-|\varphi_{n+1}^*(z)|^2\right)^2}\right)\\
\nonumber
&=\frac{1}{\pi\left(1-|z|^2\right)^2}\left(1- \frac{ (1-|z|^2)^2 \left|\left(\frac{\varphi_{n+1}(z)}{\varphi_{n+1}^*(z)}\right)^{\prime}\right|^2}{\left(\left|\frac{\varphi_{n+1}(z)}{\varphi_{n+1}^*(z)}\right|^2-1
\right)^2}\right)\\
\nonumber
&=\frac{1-|k_n(z)|^2}{\pi(1-|z|^2)^2},
\end{align}
where
\begin{equation*}
k_n(z)=\frac{(1-|z|^2)b_n^{\prime}(z)}{1-|b_n(z)|^2}, \quad b_n(z)=\frac{\varphi_{n+1}(z)}{\varphi_{n+1}^*(z)},
\end{equation*}
and hence completes the proof of \eqref{hnOPUC} in Theorem \ref{OP}.
\end{proof}

\subsection{The Limiting Value of the Intensity Function for Random Orthogonal Polynomials Associated to the Nevai Class}

\begin{proof}[Proof of \eqref{IOA} in Theorem \ref{GOP2}]
Since the convergence of \eqref{OPRLNevai} is uniform on compact subsets away from the support of $\mu$, for $z\notin \text{supp}\ \mu$ we can differentiate to yield
\begin{align}\nonumber
\lim_{n\rightarrow \infty}a_n^{\prime}(z)&=\lim_{n\rightarrow \infty}\left(\frac{p_{n+1}(z)}{p_n(z)}\right)^{\prime}\\
\nonumber
&=\frac{d}{dz}\left( \frac{z-b+\sqrt{(z-b)^2-4a^2}}{2} \right)\\
\label{pnp}
&=\frac{z-b+\sqrt{(z-b)^2-4a^2}}{2\sqrt{(z-b)^2-4a^2}}.
\end{align}
Also from \eqref{OPRLNevai} we see that
\begin{align}\nonumber
\lim_{n\rightarrow \infty}\textup{Im}(a_n(z))&=\lim_{n\rightarrow \infty} \frac{\frac{p_{n+1}(z)}{p_n(z)}-\frac{p_{n+1}(\bar{z})}{p_n(\bar{z})}}{2i}\\
\label{Impn}
&=\frac{z-\sqrt{(z-b)^2-4a^2}-(\bar{z}+\sqrt{(\bar{z}-b)^2-4a^2}}{4i}
\end{align}
Combining \eqref{pnp} and \eqref{Impn} gives
\begin{align*}
\lim_{n\rightarrow \infty}h_n(z)^2&=\lim_{n\rightarrow \infty}\frac{\left(\textup{Im}(z)\right)^2|a_n^{\prime}(z)|^2}{\left(\textup{Im}(a_n(z))\right)^2}=\frac{\left(\textup{Im}(z)\right)^2| z-b+\sqrt{(z-b)^2-4a^2} |^2}{|(z-b)^2-4a^2|(\textup{Im}(z+\sqrt{(z-b)^2-4a^2}\ ) )^2}.
\end{align*}
Therefore, using the representation of the intensity function in \eqref{hnOPRL} of Theorem \ref{OP}, from the above limit we see that
\begin{align*}
\lim_{n\rightarrow \infty}\rho_n(z)&=\lim_{n\rightarrow \infty}\frac{1-h_n^2(z)}{4\pi \left(\textup{Im}(z)\right)^2}\\
&=\frac{1}{4\pi \left(\emph{\textup{Im}}(z)\right)^2}-\frac{|z-b+\sqrt{(z-b)^2-4a^2}|^2}{4\pi|(z-b)^2-4a^2|( \emph{\textup{Im}}(z+\sqrt{(z-b)^2-4a^2}\ ))^2},
\end{align*}
locally uniformly for $z \notin \text{supp}\ \mu$, and thus completes the proof.
\end{proof}

\begin{proof}[Proof of \eqref{LIOPUC} in Theorem \ref{GOP2}]
Under the assumption that $\{\varphi_j\}$ are OPUC in the Nevai class, \eqref{OPUCNevai} gives
\begin{equation}\label{bn}
\lim_{n\rightarrow \infty}b_n(z)=\lim_{n\rightarrow \infty}\frac{\varphi_{n+1}(z)}{\varphi_{n+1}^*(z)}=0,
\end{equation}
uniformly on compact subsets of $\D$.  Since the convergence is locally uniform in $\D$, within $\D$ we can differentiate to achieve
\begin{equation}\label{bnp}
\lim_{n\rightarrow \infty}b_n^{\prime}(z)=\lim_{n\rightarrow \infty}\frac{d}{dz}\left(\frac{\varphi_{n+1}(z)}{\varphi_{n+1}^*(z)}\right)=0.
\end{equation}
Thus combining \eqref{bn} and \eqref{bnp} we see that
\begin{equation}
\lim_{n\rightarrow \infty} k_n(z)=\lim_{n\rightarrow \infty}\frac{(1-|z|^2)b_{n}^{\prime}(z)}{1-|b_{n}(z)|^2}=0.
\end{equation}
This gives that the intensity function in Theorem \ref{OP} represented by \eqref{hnOPUC} satisfies
\begin{equation*}
\lim_{n\rightarrow \infty}\rho_n(z)=\lim_{n\rightarrow \infty}\frac{1-|k_{n}(z)|^2}{\pi (1-|z|^2)^2}=\frac{1}{(1-|z|^2)^2}
\end{equation*}
locally uniformly on $\D$.

To see that the same limit holds in the exterior of the disk, observe that for $z^{-1}\in \D\setminus \{0\}$
\begin{equation}\label{phnout}
0=\lim_{n\rightarrow \infty}b_n\left(\frac{1}{z}\right)=\lim_{n\rightarrow \infty}\frac{\varphi_{n+1}(z^{-1})}{\varphi_{n+1}^*(z^{-1})}=\lim_{n\rightarrow \infty}\frac{\varphi_{n+1}(z^{-1})}{z^{-n}\varphi_{n+1}(z)}=\lim_{n\rightarrow \infty}\frac{\varphi_{n+1}^*(z)}{\varphi_{n+1}(z)},
\end{equation}
where on the second equality we have appealed to the hypothesis that the recurrence coefficients for the OPUC are such that $\{\alpha_j\}\subset (-1,1)$ so that the coefficients of $\{\varphi_j\}$ are real.

Notice that from \eqref{facrho} we can factor in a different manor to achieve
\begin{align*}
\rho_n(z)&=\frac{1}{\pi(1-|z|^2 )^2}\left(1-\frac{(1-|z|^2)^2|\varphi_{n+1}^{*}(z)\varphi_{n+1}^{\prime}(z)-\varphi_{n+1}^{*\ \prime}(z)\varphi_{n+1}(z)  |^2}{(|\varphi_{n+1}(z)|^2-|\varphi_{n+1}^*(z)|^2)^2}\right)\\
&=\frac{1}{\pi(1-|z|^2)^2}\left(1- \frac{ (1-|z|^2)^2 \left|\left(\frac{\varphi_{n+1}^*(z)}{\varphi_{n+1}(z)}\right)^{\prime}\right|^2}{\left(\left|\frac{\varphi_{n+1}^*(z)}{\varphi_{n+1}(z)}\right|^2-1
\right)^2}\right)\\
&=\frac{1-|l_n(z)|^2}{(1-|z|^2)^2},
\end{align*}
where
\begin{equation*}
l_n(z)=\frac{(1-|z|^2)c_n^{\prime}(z)}{1-|c_n(z)|^2}, \quad c_n(z)=\frac{\varphi_{n+1}^*(z)}{\varphi_{n+1}(z)}.
\end{equation*}
Using \eqref{phnout} and continuing analogously as done for the case in the unit circle, it follows that $l_n(z)\rightarrow 0$ locally uniformly for $z\in \C\setminus \overline{\D}$ as $n\rightarrow \infty$. Therefore
$$\lim_{n\rightarrow \infty}\frac{1-|l_n(z)|^2}{(1-|z|^2)^2}=\frac{1}{(1-|z|^2)^2}$$
uniformly on compact subsets of $\C\setminus \overline{\D}$, and hence gives our desired result.
\end{proof}

\subsection{Zeros of Random Orthogonal Polynomials spanned by OPUC in Shrinking Neighborhoods of the Unit Circle}

To prove Theorem \ref{OPUCBand} we will rely on a universality result by Levin and Lubinsky \cite{LevLub07}.  One of the hypothesis of their result requires that the measure $\mu$ associated to the OPUC $\{\varphi_j\}$ is regular in the sense of Ullman-Stahl-Totik, that is
\begin{equation}
\lim_{n\rightarrow \infty}\frac{\log |\kappa_n|}{n}=0,
\end{equation}
where $\kappa_n$ is the leading coefficient of $\varphi_n(z)$. We note that using equation 1.5.22 of \cite{BS}, it follows that
\begin{equation}\label{kn}
\kappa_n=\prod_{j=0}^{n}(1-\alpha_j^2)^{-1/2}.
\end{equation}
In our hypothesis of Theorem \ref{OPUCBand}, since we are assuming that the recurrence coefficients associated to $\{\varphi_j\}$ satisfy $ \alpha_j\to 0 $ as $ j \to\infty $, appealing to \eqref{kn} we see that
\begin{equation}\label{SUT}
\lim_{n\rightarrow \infty}\frac{\log |\kappa_n|}{n}=\lim_{n\rightarrow \infty}\frac{-\frac{1}{2}\sum_{j=0}^n\log |1-\alpha_j^2|}{n}=0
\end{equation}
so that the measure $ \mu $ is regular in the sense of Ullman-Stahl-Totik. For convenience of the reader, the result by Levin and Lubinsky we will use is the following:
\begin{theorem}[Theorem 6.3 Levin and Lubinsky \cite{LevLub07}]\label{LL}
Let $\mu$ be a finite positive Borel measure on $[-\pi,\pi)$ that is Ullman-Stahl-Totik regular.  Let $J\subset (-\pi,\pi)$ be compact, and such that $\mu$ is absolutely continuous in an open interval containing $J$.  Assume moreover, that $w=\mu^{\prime}$ is positive and continuous at each point of $J$.  Then uniformly for $a,b$ in compact subsets of the plane and $z=e^{i\theta}$, $\theta \in J$ and  we have
$$\lim_{n\rightarrow \infty}\frac{K_n\left(z\left(1+\frac{i2\pi a}{n}\right),z\left(1+\frac{i2\pi \overline{b}}{n}\right)\right)}{K_n(z,z)}=e^{i\pi(a-b)}\frac{\sin \pi (a-b)}{\pi(a-b)}.$$
\end{theorem}
Changing the variables in by $a=u/(2\pi i)$ and $\bar{b}=\bar{v}/(2\pi i)$, the conclusion of the above result can be restated as
$$\lim_{n\rightarrow \infty}\frac{K_n\left(z\left(1+\frac{u}{n}\right),z\left(1+\frac{\overline{v}}{n}\right)\right)}{K_n(z,z)}= \frac{e^{u+v}-1}{u+v}:=H(u+v).$$

\begin{proof}[Proof of Theorem \ref{OPUCBand}]
It follows from definition of the intensity function that
\begin{eqnarray*}
\frac1n \E\big[N_n\big(\Omega(S,\tau_1,\tau_2)\big)\big] & = & \frac1n\iint_{\Omega(S,\tau_1,\tau_2)}\rho_n(z)\ d A \\
& = & \frac{1}{n}\int_S \int_{1+\frac{\tau_1}{2n}}^{1+\frac{\tau_2}{2n}}\rho_n(zr)r dr|dz|\\
& = & \frac1{2n^2}\int_S\int_{\tau_1}^{\tau_2}\rho_n\left(z\left(1+\frac\tau{2n}\right)\right)\left(1+\frac\tau{2n}\right)\ d\tau| d z|.
\end{eqnarray*}
Since
\begin{equation*}
\frac{1}{2}\int_{S}|dz|=\frac{|S|}{2},
\end{equation*}
and as $n\rightarrow \infty$ we have $1+\tau/(2n)\rightarrow 1$ uniformly for $\tau$ on compact subsets of the real line,
to complete the the proof it suffices to show
\begin{equation}
\label{needed1}
\lim_{n\rightarrow \infty}\frac1{n^2}\rho_n\left(z\left(1+\frac\tau{2n}\right)\right)= \frac1{\pi}\left(\frac{H^\prime(\tau)}{H(\tau)}\right)^\prime
\end{equation}
uniformly for $ z\in S $ and $ \tau $ on compact subsets of the real line.

Under the conditions of the hypothesis, given \eqref{SUT} we can use Theorem \ref{LL} to achieve
\begin{equation}
\label{scaling-limit-1}
\lim_{n\to\infty} K_n(z_{n,u},z_{n,\overline v})K_n^{-1}(z,z) = H(u+v)
\end{equation}
uniformly for $ z \in S $ and $ u,v $ on compact subsets of $ \C $, where $ z_{n,a}:=z(1+a/n) $. Since the above convergence is uniform for $z\in S$ and $u,v$ on compact subsets of $\C$, we can differentiate to yield
\begin{equation}
\label{scaling-limit-2}
\lim_{n\rightarrow \infty}\frac{\frac{\partial^{i+j}}{\partial u^i\partial v^j} K_n(z_{n,u},z_{n,\overline v})}{K_n(z,z)}
\lim_{n\to\infty} \frac{z^{i-j}}{n^{i+j}}\frac{K_n^{(i,j)}(z_{n,u},z_{n,\overline v})}{K_n(z,z)} = H^{(i+j)}(u+v)
\end{equation}
where we retain the convergence uniformly for $ z \in S $ and $ u,v $ on compact subsets of $ \C $.

Therefore, using the representation \eqref{thm2.1} of $\rho_n$ and the two limits \eqref{scaling-limit-1} and \eqref{scaling-limit-2} gives
\begin{align*}
\frac{1}{n^2}\rho_n\left(z\left(1+\frac\tau{2n}\right)\right)&= \frac{1}{n^2 \pi} \frac{K_n(z_{n,\tau/2},z_{n,\bar{\tau}/2})K_n^{(1,1)}(z_{n,\tau/2},z_{n,\bar{\tau}/2})-|K_n^{(0,1)}(z_{n,\tau/2},z_{n,\bar{\tau}/2})|^2}{K_n(z_{n,\tau/2},z_{n,\bar{\tau}/2})^2}\\
&=\frac{1}{\pi}\frac{\frac{K_n(z_{n,\tau/2},z_{n,\bar{\tau}/2})K_n^{(1,1)}(z_{n,\tau/2},z_{n,\bar{\tau}/2})}{n^2K_n(z,z)^2}-\frac{
|K_n^{(0,1)}(z_{n,\tau/2},z_{n,\bar{\tau}/2})|^2}{n^2K_n(z,z)^2}}{\frac{K_n(z_{n,\tau/2},z_{n,\bar{\tau}/2})^2}{K_n(z,z)^2}}\\
&\rightarrow \frac{1}{\pi}\frac{H(\tau)H^{\prime \prime}(\tau)-H^{\prime}(\tau)^2}{H(\tau)^2} \quad (n\rightarrow \infty) \\
&=\frac1{\pi}\left(\frac{H^\prime(\tau)}{H(\tau)}\right)^\prime,
\end{align*}
and thus completes the proof.
\end{proof}

\textbf{Acknowledgements.}  The work presented is a portion of the author's Ph.~D.~thesis under the supervision of Igor Pritsker.
The author is grateful to Igor Pritsker for the many helpful discussions and for financial support through his grant from the National Security Agency.  Financial support was also provided by the Jeanne LeCaine Agnew Endowed Fellowship and by the Vaughn Foundation on behalf of Anthony Kable.  The author would also like to thank Maxim Yattselev for bringing the author's attention to the ratio asymptotics that have greatly simplified the proofs of the limiting value of the intensity functions, and for making the author aware of the techniques used in the proof of Theorem \ref{OPUCBand}.

\end{document}